\documentclass[12pt, twoside]{article}
\usepackage{amsmath,amsthm,amssymb}
\usepackage{times}
\usepackage{enumerate}

\def\titlerunning#1{\gdef\titrun{#1}}
\makeatletter
\def\author#1{\gdef\autrun{\def\and{\unskip, }#1}\gdef\@author{#1}}
\def\address#1{{\def\and{\\\hspace*{18pt}}\renewcommand{\thefootnote}{}%
\footnote {#1}}%
\markboth{\autrun}{\titrun}}
\makeatother
\def\email#1{e-mail: #1}
\def\subjclass#1{{\renewcommand{\thefootnote}{}%
\footnote{\emph{Mathematics Subject Classification (2010):} #1}}}
\def\keywords#1{\par\medskip
\noindent\textbf{Keywords.} #1}

\newtheorem{thm}{Theorem}[section]
\newtheorem{corollary}[thm]{Corollary}
\newtheorem{proposition}[thm]{Proposition}
\newtheorem{lemma}[thm]{Lemma}

\theoremstyle{definition}

\newtheorem{remark}[thm]{Remark}

\numberwithin{equation}{section}

\def\CP1{\mathbb{C}\mathrm{P}^1}

\def\dim{{\rm dim}}
\def\diag{{\rm diag}}
\def\Id{{\rm Id}}

\def\td{{\widetilde d}}
\def\mcM{{\mathcal M}}

\def\oM{{\overline{\mathcal{M}}}}
\def\mbC{{\mathbb C}}
\def\mbP{{\mathbb P}}

\def\tm{{\widetilde m}}
\def\m{{\widetilde{m}}}

\def\CP{{{\mathbb C}{\rm P}}}

\def\mbQ{{\mathbb Q}}

\begin{document}


\baselineskip=17pt


\titlerunning{Top tautological group of $\mcM_{g,n}$}

\title{Top tautological group of $\mcM_{g,n}$}

\author{Alexandr Buryak \and Sergey Shadrin \and Dimitri Zvonkine}

\date{}

\maketitle

\address{A. Buryak: Department of Mathematics, ETH Zurich, Ramistrasse 101 8092, Zurich, Switzerland; \email{buryaksh@gmail.com}, corresponding author
\and
S. Shadrin: Korteweg-de Vries Institute for Mathematics, University of Amsterdam, P.~O.~Box 94248, 1090 GE Amsterdam, The Netherlands; \email{s.shadrin@uva.nl}
\and
D. Zvonkine: Institut Math\'{e}matique de Jussieu, CNRS and UPMC, 4 place Jussieu, 75013 Paris, France; \email{dimitri.zvonkine@gmail.com}
}



\begin{abstract}
We describe the structure of the top tautological group in the cohomology of the moduli space of smooth genus $g$ curves with $n$ marked points.   
\end{abstract}

\section{Introduction}

In this paper we study the cohomology groups of the moduli space of smooth genus $g$ curves with $n$ marked points. This moduli space is denoted by $\mcM_{g,n}$. The space $\mcM_{g,0}$ will be denoted by $\mcM_g$. 

The cohomology of $\mcM_{g,n}$ has a distinguished subring of tautological classes 
$$
R^*(\mcM_{g,n})\subset H^{\rm even} (\mcM_{g,n};\mbQ)
$$
studied extensively since Mumford's seminal article \cite{Mum83}. 

A great step towards an understanding of the tautological ring of $\mcM_g$ was done by C. Faber in \cite{Fab99}. He formulated three conjectures that give a complete description of~$R^*(\mcM_g)$. These conjectures are called the socle conjecture, the top intersection conjecture and the perfect pairing conjecture. The socle conjecture was proved in \cite{Loo95}, there are several proofs of the top intersection conjecture (see \cite{GP98,LX09,BS11}). The perfect pairing conjecture is true up to genus $23$, but the accumulating evidence suggests it may be wrong for $g \geq 24$. 

The socle conjecture says that $R^{>g-2}(\mcM_g)=0$ and $R^{g-2}(\mcM_g)=\mbQ$. We will recall the other two conjectures in Section~\ref{subsection: Faber's conjectures}.

Analogous statements can be formulated about the tautological ring of $\mcM_{g,1}$. The socle property in this case says that $R^{>g-1}(\mcM_{g,1})=0$ and $R^{g-1}(\mcM_{g,1})=\mbQ$ (\cite{Loo95}). The top intersection property for $\mcM_{g,1}$ can be easily derived from the original top intersection statement for $\mcM_g$.  The perfect pairing conjecture in this case is also open.

In this paper we formulate and prove a socle and intersection numbers property for~$\mcM_{g,n}$. In particular, the generalized socle property says that $R^{>g-1}(\mcM_{g,n})=0$ and $R^{g-1}(\mcM_{g,n})=\mbQ^n$. The vanishing part $R^{>g-1}(\mcM_{g,n})=0$ was actually proved in \cite{Ion02}. 

Let us say a few words about the idea of our proof. One can choose different spanning families in the tautological groups. On one hand the tautological groups of $\mcM_{g,n}$ are spanned by monomials in $\psi$-classes and $\kappa$-classes. On the other hand the tautological groups of $\mcM_{g,n}$ are spanned by double ramification cycles. A technique for working with these cycles was developed, e.~g., in papers \cite{Ion02,Sha03,SZ08,BSSZ12}. In \cite{Ion02} it was used for proving the vanishing $R^{>g-1}(\mcM_{g,n})=0$ and in \cite{Sha03} it was used for studying the intersection theory of the moduli space of curves. This technique was also applied in \cite{BS11} in order to give a simple proof of Faber's top intersection conjecture. In this paper we develop the theory of double ramification cycles and use it for the proof of generalized socle and top intersection properties.

\subsection{Tautological ring} 

In this section we briefly recall basic definitions related to the tautological ring of the moduli space of curves. We refer to \cite{Vak08,Zvo12} for a more detailed introduction in this subject.

Let $\oM_{g,n}$ be the moduli space of stable genus $g$ curves with $n$ marked points. The class $\psi_i\in H^2(\oM_{g,n};\mbQ)$ is defined as the first Chern class of the line bundle over $\oM_{g,n}$ formed by the cotangent lines at the $i$-th marked point. Let $\pi\colon\oM_{g,n+1}\to\oM_{g,n}$ be the forgetful map that forgets the last marked point. The class $\kappa_k\in H^{2k}(\oM_{g,n};\mbQ)$ is defined as follows: 
$$
\kappa_k:=\pi_*(\psi_{n+1}^{k+1}).
$$ 

It is convenient to define multi-index kappa classes. Let $m\ge 1$ and consider the forgetful map $\pi\colon\oM_{g,n+m}\to\oM_{g,n}$ that forgets the last $m$ marked points. Let $k_1,k_2,\ldots,k_m$ be non-negative integers. Define the class $\kappa_{k_1,k_2,\ldots,k_m}\in H^{2\sum_{i=1}^m k_i}(\oM_{g,n};\mbQ)$~by 
$$
\kappa_{k_1,k_2,\ldots,k_m}:=\pi_*\left(\psi_{n+1}^{k_1+1}\psi_{n+2}^{k_2+1}\ldots\psi_{n+m}^{k_m+1}\right).
$$
Multi-index $\kappa$-classes can be expressed as polynomials in $\kappa$-classes with one index. Conversely, any polynomial in one index $\kappa$-classes can be written as a linear combination of multi-index $\kappa$-classes.

The tautological ring $R^*(\mcM_{g,n})$ is defined as the subring of $H^*(\mcM_{g,n};\mbQ)$ generated by the classes $\kappa_j$ and $\psi_i$. The group $R^i(\mcM_{g,n})$ is defined by $R^i(\mcM_{g,n}):=R^*(\mcM_{g,n})\cap H^{2i}(\mcM_{g,n};\mbQ)$.

\subsection{Faber's conjectures}\label{subsection: Faber's conjectures}

Here we recall Faber's conjectures from~\cite{Fab99} about the structure of the tautological ring $R^*(\mcM_g)$. Let $g\ge 2$.

\begin{itemize}

\item(\emph{Socle}) $R^{>g-2}(\mcM_g)=0$ and $R^{g-2}(\mcM_g)=\mbQ$.

\item(\emph{Top intersection}) Suppose that $k_1+\cdots+k_m=g-2$ and $k_i\ge 0$. Then we have the following equation in $R^{g-2}(\mcM_g)$:   
$$
\kappa_{k_1,\ldots,k_m}=\frac{(2g-3+m)!(2g-3)!!}{(2g-2)!\prod_{i=1}^m (2k_i+1)!!}
\kappa_{g-2}.
$$ 

\item(\emph{Perfect pairing}) For any $0\leq i\leq g-2$, the cup product defines the pairing 
$$
R^i(\mcM_g)\times R^{g-2-i}(\mcM_g)\to R^{g-2}(\mcM_g)=\mbQ.
$$ 
This pairing is non-degenerate.

\end{itemize}

It is easy to see that Faber's conjectures, if true, completely determine the structure of the tautological ring $R^*(\mcM_g)$.

\subsection{Generalized Faber conjectures}\label{subsection: Zvonkine's conjectures} In this section we formulate analogous properties of the tautological ring $R^*(\mcM_{g,n})$. Assume that $g\ge 2$ and $n\ge 1$.

\begin{itemize}

\item\emph({Generalized socle}) $R^{>g-1}(\mcM_{g,n})=0$ and $R^{g-1}(\mcM_{g,n})=\mbQ^n$. The classes $\psi_i^{g-1}$, $i=1,2,\ldots,n$, form a basis in $R^{g-1}(\mcM_{g,n})$.

\item\emph({Generalized top intersection}) Suppose that $d_1+\cdots+d_n+k_1+\ldots+k_m=g-1$ and $d_i,k_j\ge 0$. Then we have the following equation in $R^{g-1}(\mcM_{g,n})$:  
\begin{align*}
\prod_{i=1}^n\psi_i^{d_i}\cdot\kappa_{k_1,k_2,\ldots,k_m}=&\frac{(2g-1)!!}{\prod_{i=1}^n(2d_i+1)!!\prod_{j=1}^m(2k_j+1)!!}\frac{(2g-3+n+m)!}{(2g-2+n)!}\times\\
&\times\sum_{i=1}^n\frac{(2g-2+n)d_i + \sum k_j}{g-1}\psi_i^{g-1}.
\end{align*}

\item\emph({Generalized perfect pairing}) The ring $R^*(\mcM_{g,n})$ is level of type $n$. In other words, a polynomial in  $\psi_1,\ldots,\psi_n,\kappa_1,\ldots,\kappa_{g-1}$ vanishes, if and only if its products with all classes of complementary dimension vanish in $R^{g-1}(\mcM_{g,n})$.

\end{itemize}

Similarly to Faber's conjectures, these properties, if true, completely determine the structure of the ring $R^*(\mcM_{g,n})$. 

In this paper we prove the generalized socle and top intersection properties. As for the perfect pairing property, it is true in many cases that can be checked on computer, but recent evidence leads to serious doubts that it is valid in general. The main result of this paper is the following theorem.
\begin{thm}
The generalized socle and top intersection properties are true.
\end{thm}

\subsection{Organization of the paper}

In Section~\ref{section: generalized top intersection} we prove the generalized top intersection property, assuming that the generalized socle conjecture is true. We also show that $\dim R^{g-1}(\mcM_{g,n}) \geq n$ (the easy part of the socle property). The rest of the paper is devoted to the hard part of the socle property, that is, the inequality $\dim R^{g-1}(\mcM_{g,n}) \leq n$.

In Section~\ref{section: three statements} we formulate three main ingredients of the proof of the generalized socle property: Lemma~\ref{lemma: symmetry}, Proposition~\ref{proposition: first step} and Proposition~\ref{proposition: second step}. We show how the socle property follows from them. These statements will be proved in the subsequent sections and the proof of the last proposition is the hardest one. 

In Section~\ref{section: DR-cycles} we introduce the main tool for proving these statements -- double ramification cycles. 

Section \ref{section: technical lemmas} contains several linear algebra lemmas that simplify the proofs of Propositions \ref{proposition: first step} and \ref{proposition: second step}. 

In Section~\ref{section: main proofs} we prove Lemma~\ref{lemma: symmetry}, Proposition~\ref{proposition: first step} and Proposition~\ref{proposition: second step}.


\section{Generalized top intersection}\label{section: generalized top intersection}

In this section we show that the classes $\psi_i^{g-1}$ are linearly independent in $R^{g-1}(\mcM_{g,n})$ and, thus, $\dim R^{g-1}(\mcM_{g,n}) \geq n$. Then we prove the generalized  top intersection property assuming the equality 
$\dim R^{g-1}(\mcM_{g,n}) = n$. This equality will be proved in the subsequent sections.

\begin{proposition}
The classes $\alpha_s:=\lambda_g \lambda_{g-1} \psi_1 \cdots \widehat{\psi_s} \cdots \psi_n$ (where the hat means a missing factor) vanish on the boundary of $\oM_{g,n}$.
\end{proposition}

\begin{proof}
It is well-known (see~\cite{Fab97}, Lemma~1) that $\lambda_g \lambda_{g-1}$ vanishes on $\oM_{g,n} \setminus \mcM_{g,n}^{\rm rt}$, where $\mcM_{g,n}^{\rm rt}$ is the space of stable curves with one genus~$g$ component and possibly several ``rational tails'' composed of genus~0 components. Thus, it remains to show that the classes~$\alpha_s$ also vanish on $\mcM_{g,n}^{\rm rt} \setminus \mcM_{g,n}$. Every boundary divisor in $\mcM_{g,n}^{\rm rt}$ is isomorphic to a product $\mcM^{\rm rt}_{g,n_1+1} \times \oM_{0,n_2+1}$, where $n_1+n_2=n$. Among the $\psi$-classes that make part of $\alpha_s$, at least $n_2-1$ are sitting on the second factor. Since the dimension of $\oM_{0,n_2+1}$ equals $n_2-2$, we see that the class $\alpha_s$ vanishes on our boundary divisor for dimension reasons.
\end{proof} 

Let $\pi\colon\mcM^{\rm rt}_{g,N}\to\mcM_g$ be the forgetful map. Let $l_1,l_2,\ldots,l_N$ be non-negative integers such that $l_1+l_2+\ldots+l_N=g-2$. Recall that Faber's top intersection conjecture says that
\begin{gather}\label{eq: faber's formula}
\pi_*(\psi_1^{l_1+1}\ldots\psi_N^{l_N+1})=\frac{(2g-3+N)!(2g-3)!!}{(2g-2)!\prod_{i=1}^N (2l_i+1)!!}
\kappa_{g-2}.
\end{gather}
The following small generalization of this formula will be useful for us. Define $(-1)!!:=1$.
\begin{lemma}\label{lemma: auxiliary}
Let $l_1,\ldots,l_N$ be integers such that $l_1+\ldots+l_N=g-2$. Suppose that at most one of $l_1,\ldots,l_N$ is equal to $-1$ and the others are non-negative. Then formula~\eqref{eq: faber's formula} holds.   
\end{lemma}
\begin{proof}
If all $l_i$'s are non-negative, then equation~\eqref{eq: faber's formula} is exactly Faber's top intersection conjecture.

Suppose one of $l_i$'s is equal to $-1$. We proceed by induction on $N$. If $N=1$, then $l_1=g-2\ge 0$, so the formula is true. Suppose~$N\ge 2$. Without loss of generality we can assume that $l_1=-1$. Using the string equation and the induction assumption, we get
\begin{multline*}
\pi_*(\psi_2^{l_2+1}\ldots\psi_N^{l_N+1})=\sum_{i=2}^N\frac{(2g-4+N)!(2g-3)!!}{(2g-2)!(2l_i-1)!!\prod_{\substack{2\le j\le N\\j\ne i}} (2l_j+1)!!}\kappa_{g-2}=\\
=\frac{(2g-3+N)!(2g-3)!!}{(2g-2)!\prod_{i=1}^N (2l_i+1)!!}\kappa_{g-2}.
\end{multline*}
The lemma is proved.
\end{proof}

Let $d_1,\dots,d_n$ and $k_1, \dots, k_m$ be nonnegative integers. Assume that $\sum d_i + \sum k_i = g-1$. These integers will be fixed for the rest of the section. Denote
$$
C:=\frac{(2g-3+n+m)! \, (2g-3)!!}
{(2g-2)! \, \prod_{i=1}^n (2d_i +1)!! \, \prod_{j=1}^m (2k_j+1)!!}.
$$

\begin{lemma} \label{Lem:pushforward}
Let $\pi\colon\mcM^{\rm rt}_{g,n} \to \mcM_g$ be the forgetful map. Then in $R^{g-2}(\mcM_g)$ we have
$$
\pi_* \left(
\psi_1^{d_1+1} \cdots \psi_s^{d_s} \cdots \psi_n^{d_n+1} \kappa_{k_1, \dots, k_m}
\right)
= C \cdot (2d_s+1) \cdot
\kappa_{g-2}.
$$
\end{lemma}

\begin{proof}
Let $\pi'\colon\mcM_{g,n+m}^{\rm rt}\to\mcM_g$ be the forgetful map. We have
\begin{multline*}
\pi_* \left(
\psi_1^{d_1+1} \cdots \psi_s^{d_s} \cdots \psi_n^{d_n+1} \kappa_{k_1, \dots, k_m}
\right)
= \pi'_* \left(\psi_1^{d_1+1} \cdots \psi_s^{d_s} \cdots \psi_n^{d_n+1} \psi_{n+1}^{k_1+1} \cdots \psi_{n+m}^{k_m+1}
\right)=\\
\stackrel{\text{by Lemma \ref{lemma: auxiliary}}}{=}C \cdot (2d_s+1) \cdot
\kappa_{g-2}.
\end{multline*}
\end{proof}

Denote by $A_g$ the nonzero intersection number
$$
A_g:=\int\limits_{\oM_g} \kappa_{g-2} \lambda_g \lambda_{g-1} = 
\frac{(-1)^{g-1}B_{2g}(g-1)!}{2^g (2g)!},
$$
where $B_{2g}$ is the Bernoulli number (see~\cite{Fab97}, Lemma~2).

\begin{corollary} \label{Cor:integral}
We have $\int_{\oM_{g,n}} \prod_{i=1}^n \psi_i^{d_i}
\kappa_{k_1, \dots, k_m}  \cdot \alpha_s = 
C\cdot (2d_s+1) \cdot  A_g$.
\end{corollary}
\begin{proof}
Compute the integral by first projecting on $\oM_g$ and use Lemma~\ref{Lem:pushforward}.
\end{proof}

\begin{proposition}
The $n \times n$ matrix $M_{is}:= \int_{\oM_{g,n}} \psi_i^{g-1} \alpha_s$ is non-degenerate.
\end{proposition}

\begin{proof}
Denote by $U$ the $n \times n$ matrix given by $U_{is} = 1$ for all $i,s$. It has exactly two eigenvalues: 0 and $n$. Therefore $U+a\Id$ is non-degenerate whenever $a \ne 0, -n$.

According to the corollary, we have $M_{ii} = C \cdot A_g \cdot (2g-1)$ and $M_{is} = C \cdot A_g$, for~$i \ne s$. Thus,
$$
M = C \cdot A_g \cdot \left(U + (2g-2)\Id \right),
$$
so it is non-degenerate.
\end{proof}

\begin{proposition}
The classes 
$$
\prod_{i=1}^n \psi_i^{d_i} \kappa_{k_1, \dots, k_m}
$$ 
and
$$
C\frac{(2g-1)!}{(2g-2+n)!}\sum_{i=1}^n\frac{(2g-2+n)d_i + \sum k_j}{g-1}\psi_i^{g-1}
$$
have the same intersection number with every $\alpha_s$.
\end{proposition}

\begin{proof}
Apply Corollary~\ref{Cor:integral} and divide both intersection numbers by the common factor~$C A_g$. We obtain the following equality that needs to be checked:
\begin{align*}
2d_s+1 \stackrel{?}{=}& 
\frac{2g-1}{(2g-2+n)!}
\sum_{i=1}^n  \left[\frac{(2g-2+n) d_i + \sum k_j}{g-1} \cdot
\frac{(2g-3+n)!}{2g-1}(1 + (2g-2)\delta_{is})
\right]=\\
=&\frac1{(2g-2+n)(g-1)}
\sum_{i=1}^n \left((2g-2+n)d_i + \sum k_j \right)
(1 + (2g-2)\delta_{is})=\\
=&\frac1{(2g-2+n)(g-1)}
\left[ (2g-2+n) \left(\sum d_i + \sum k_j\right) + (2g-2+n)(2g-2) d_s
\right]=\\
=&\frac1{g-1} [(g-1) + (2g-2) d_s] = 2d_s+1.
\end{align*}
Thus, the equality is, indeed, right. 
\end{proof}

Let us sum up the results of our computations. We have found $n$ classes $\alpha_s$, $1 \leq s \leq n$, of degree $2g+n-2$ that vanish on the boundary of $\oM_{g,n}$ and, thus, can be used as linear forms on the group $R^{g-1}(\mcM_{g,n})$. We have proved that the intersection matrix of the classes~$\psi_i^{g-1}$ and $\alpha_s$ is non-degenerate and, therefore, $\dim R^{g-1} (\mcM_{g,n}) \geq n$. Finally, we have computed the intersection numbers of all tautological classes in $R^{g-1}(\mcM_{g,n})$ with the classes~$\alpha_s$. {\em Assuming} that $(\psi_i^{g-1})$ is a basis of $R^{g-1}(\mcM_{g,n})$, this allowed us to decompose any class in this basis, thus, proving the generalized top intersection property.


\section{Generalized socle property: three statements}\label{section: three statements}

In this section we formulate three statements and show how to use them for the proof of the generalized socle property. We will prove these statements in the next sections.

\subsection{Three statements}

Denote by $\pi_k\colon\mcM_{g,n}\to\mcM_{g,n-1}$ the forgetful map that forgets the $k$-th marked point. Let $i_{k,l}\colon\mcM_{g,n}\to\mcM_{g,n}$ be the map that interchanges the $k$-th and the $l$-th marked points.

Let $S^j_{k,l}(\mcM_{g,n})$ be the subspace of $R^j(\mcM_{g,n})$ defined by
$$
S^j_{k,l}(\mcM_{g,n}):=\{\alpha\in R^j(\mcM_{g,n})|i_{k,l}^*\alpha=\alpha\}.
$$

\begin{lemma}\label{lemma: symmetry}
Let $n\ge 2$ and $1\le k<l\le n$, then we have 
$$
S^{g-1}_{k,l}(\mcM_{g,n})\subset\pi_k^*(R^{g-1}(\mcM_{g,n-1}))+\pi_l^*(R^{g-1}(\mcM_{g,n-1})).
$$
\end{lemma}

\begin{proposition}\label{proposition: first step}
For any $n\ge 1$, we have
$$
R^{g-1}(\mcM_{g,n})=\pi_1^*(R^{g-1}(\mcM_{g,n-1}))+\psi_1\pi_1^*(R^{g-2}(\mcM_{g,n-1})).
$$
\end{proposition}

\begin{proposition}\label{proposition: second step}
For any $n\ge 1$, we have
$$
R^{g-2}(\mcM_{g,n})=\pi_1^*(R^{g-2}(\mcM_{g,n-1}))+\psi_1\pi_1^*(R^{g-3}(\mcM_{g,n-1}))+\sum_{1\le k<l\le n}S^{g-2}_{k,l}(\mcM_{g,n}).
$$
\end{proposition}

\subsection{Proof of the generalized socle conjecture}
In Section~\ref{section: generalized top intersection} we have proved the inequality $\dim R^{g-1}(\mcM_{g,n})\ge n$. It remains to prove that $\dim R^{g-1}(\mcM_{g,n})\le n$

Let $p_i\colon\mcM_{g,n}\to\mcM_{g,1}$ be the forgetful map that forgets all marked points except the~$i$-th. Since $\dim R^{g-1}(\mcM_{g,1})=1$ (\cite{Loo95}), it is sufficient to prove that, for any $n\ge 1$, the group $R^{g-1}(\mcM_{g,n})$ is spanned by the pullbacks $p_i^*(R^{g-1}(\mcM_{g,1}))$, where $1\le i\le n$. Equivalently, we have to prove that
\begin{gather}\label{eq: tmp equation}
R^{g-1}(\mcM_{g,n})=\sum_{i=1}^n\pi_i^*(R^{g-1}(\mcM_{g,n-1})),
\end{gather}
for any $n\ge 2$.

From Propositions \ref{proposition: first step} and \ref{proposition: second step} it follows that
\begin{align*}
R^{g-1}(\mcM_{g,n})=&\pi_1^*(R^{g-1}(\mcM_{g,n-1}))+\psi_1(\pi_2\circ\pi_1)^*(R^{g-2}(\mcM_{g,n-2}))+\\
&+\psi_1\psi_2(\pi_2\circ\pi_{1})^*(R^{g-3}(\mcM_{g,n-2}))+\psi_1\pi_1^*\left(\sum_{2\le k<l\le n}S^{g-2}_{k,l}(\mcM_{g,n-1})\right).
\end{align*}
Obviously, we have
\begin{align*}
&\psi_1(\pi_2\circ\pi_1)^*(R^{g-2}(\mcM_{g,n-2}))\subset\pi_2^*(R^{g-1}(\mcM_{g,n-1})),\\
&\psi_1\psi_2(\pi_2\circ\pi_1)^*(R^{g-3}(\mcM_{g,n-2}))\subset S^{g-1}_{1,2}(\mcM_{g,n}),\\
&\psi_1\pi_1^*(S^{g-2}_{k,l}(\mcM_{g,n-1}))\subset S^{g-1}_{k,l}(\mcM_{g,n}),\text{ where $2\le k<l\le n$}.
\end{align*}
Applying Lemma \ref{lemma: symmetry} to the last two formulas, we get \eqref{eq: tmp equation}. The generalized socle conjecture is proved.


\section{Double ramification cycles}\label{section: DR-cycles}

Here we give the definition of a particular type of double ramification cycles that we need in this paper and formulate main formulas that we will use.

In Section~\ref{subsection: definition of DR-cycles} we define double ramification cycles. In Section~\ref{subsection: main formulas with DR-cycles} we prove main formulas with double ramification cycles. In Section~\ref{subsection: Hain's formula} we formulate Hain's result. In Section~\ref{subsection: DR-cycles and the tautological ring} we explain that double ramification cycles span the tautological groups of~$\mcM_{g,n}$. 

\subsection{Definition of double ramification cycles}\label{subsection: definition of DR-cycles}

Let $a_1,\dots,a_n$, $n\geq 1$, be a list of integers satisfying $\sum a_i=0$. Suppose that not all of them are equal to zero. Denote by $n_+$ the number of positive integers among the $a_i$'s. They form a partition $\mu= (\mu_1,\ldots,\mu_{n_+})$. Similarly, denote by $n_-$ the number of negative integers among the $a_i$'s. After a change of sign they form another partition $\nu =(\nu_1,\ldots,\nu_{n_-})$. Both $\mu$ and $\nu$ are partitions of the same integer 
$$
d:=\frac12\sum_{i=1}^n |a_i|.
$$
Finally, let $n_0$ be the number of vanishing $a_i$'s. 

Let $\oM_{g,n_0}(\mu,\nu)$ be the moduli space of ``rubber'' stable maps to $\mbC\mbP^1$ relative to $0$ and~$\infty$ (see e.g. \cite{GJV11, OP06}). Partitions $\mu$ and $\nu$ correspond to ramification profiles over~$0$ and~$\infty$. Denote by $\rho$ the forgetful map $\oM_{g,n_0}(\mu,\nu)\to\oM_{g,n}$.

The double ramification cycle without forgotten points is defined by 
$$
DR_g\left(\prod_{i=1}^n m_{a_i}\right):=\rho_*\left(\left[\oM_{g,n_0}(\mu,\nu)\right]^{virt}\right),
$$
where $\left[\oM_{g,n_0}(\mu,\nu)\right]^{virt}$ is the virtual fundamental class in the homology of $\oM_{g,n_0}(\mu,\nu)$ (see e.g. \cite{GJV11}).

General double ramification cycles are defined as follows. Let $k\ge 0$ and $a_1,a_2,\ldots,a_n$, $b_1,\ldots,b_k$ are integers such that not all of them are equal to zero and $\sum_{i=1}^n a_i+\sum_{j=1}^k b_j=0$. Let $\pi\colon\oM_{g,n+k}\to\oM_{g,n}$ be the forgetful map that forgets the last $k$ marked points. Then
$$
DR_g\left(\prod_{i=1}^n m_{a_i}\prod_{j=1}^k\m_{b_j}\right):=\pi_*\left(DR_g\left(\prod_{i=1}^n m_{a_i}\prod_{j=1}^k m_{b_j}\right)\right).
$$ 
In this notation the $i$-th marked point corresponds to the ramification of order $a_i$, so the order of symbols $m_{a_i}$ is important. On the other hand, the order of symbols $\m_{b_j}$ is irrelevant. Sometimes we will place them in different positions in the bracket. For example, $DR_g(m_a\m_b)$ and $DR_g(\m_b m_a)$ are, by definition, the same classes. 

The Poincare dual of $DR_g\left(\prod_{i=1}^n m_{a_i}\prod_{j=1}^k\m_{b_j}\right)$ in the cohomology of $\oM_{g,n}$ will be denoted by the same symbol. We have (see e.g. \cite{GJV11})
\begin{align*}
&DR_g\left(\prod_{i=1}^n m_{a_i}\prod_{j=1}^k\m_{b_j}\right)=0,&&\text{if $k\ge g+1$};\\
&DR_g\left(\prod_{i=1}^n m_{a_i}\prod_{j=1}^k\m_{b_j}\right)\in H^{2(g-k)}(\oM_{g,n};\mbQ),&&\text{if $k\le g$}.
\end{align*}
In \cite{FP05} it is proved that double ramification cycles belong to the tautological ring of~$\oM_{g,n}$. Below we use a lot the restriction of the classes $DR_g\left(\prod_{i=1}^n m_{a_i}\prod_{j=1}^k\m_{b_j}\right)$ to the open moduli space $\mcM_{g,n}$. Abusing notation we denote the restriction by the same symbol.

\subsection{Main formulas with double ramification cycles}\label{subsection: main formulas with DR-cycles}

Here we list main properties of double ramification cycles. 

\begin{lemma}\label{lemma: change sign in DR-cycle}
We have
\begin{gather}\label{formula: change sign in DR-cycle}
DR_g\left(\prod_{i=1}^n m_{a_i}\prod_{i=1}^k \m_{b_i}\right)=DR_g\left(\prod_{i=1}^n m_{-a_i}\prod_{i=1}^k \m_{-b_i}\right).
\end{gather}
\end{lemma}
\begin{proof}
The proof is obvious from the definition of double ramification cycles.
\end{proof}

\begin{lemma}\label{lemma: intersection with psi-class}
Let $a_1,a_2,\ldots,a_n$, $n\ge 1$, and $b_1,b_2,\ldots,b_k$, $1\le k\le g+1$, be non-zero integers such that $\sum_{i=1}^n a_i+\sum_{j=1}^k b_j=0$. Then we have the following equation in~$R^{g-k+1}(\mcM_{g,n})$:
\begin{align}\label{eq:int-psi}
\psi_1 DR_g\left(\prod_{i=1}^n m_{a_i}\prod_{j=1}^k \m_{b_j}\right)=&-\sum_{1\le i<j\le k}\frac{b_i+b_j}{ra_1}DR_g\left(\prod_{l=1}^n m_{a_l}\m_{b_i+b_j}\prod_{l\ne i,j}\m_{b_l}\right)\\
&-\sum_{i=2}^n\sum_{j=1}^k\frac{a_i+b_j}{ra_1}DR_g\left(\prod_{l=1}^n m_{a_l+\delta_{l,i}b_j}\prod_{l\ne j}\m_{b_l}\right) \notag\\
&+\sum_{j=1}^k\frac{-a_1+(r-1)b_j}{ra_1}DR_g\left(m_{a_1+b_j}\prod_{i=2}^n m_{a_i}\prod_{l\ne j}\m_{b_l}\right),\notag
\end{align}
where $r:=2g-2+n+k$.
\end{lemma}
\begin{proof}
First, observe that for $k=1$ we have a trivial identity, because $R^g(\mcM_{g,n}) = 0$. For $k=g+1$ the left-hand side of this equation is equal to zero, while the vanishing of the right-hand side follows from the substitution 
$$
DR_g\left(\prod_{i=1}^n m_{a_i}\prod_{j=1}^g \m_{b_j}\right)=g!\prod_{j=1}^g b_j^2.
$$ 

The proof of this formula in the general case is based on~\cite[Theorem 4]{BSSZ12}. Indeed, by definition, the class  $DR_g\left(\prod_{i=1}^n m_{a_i}\prod_{j=1}^k \m_{b_j}\right)$ is the restriction to $\mcM_{g,n} \subset \oM_{g,n}$ of the push-forward of the class $DR_g\left(\prod_{i=1}^n m_{a_i}\prod_{j=1}^k m_{b_j}\right)$ defined in the tautological ring of $\oM_{g,n+k}$. To prove the identity we will use the projection formula. Namely, we lift the class $\psi_1$ to $\oM_{g,n+k}$, intersect it there with the class $DR_g\left(\prod_{i=1}^n m_{a_i}\prod_{j=1}^k m_{b_j}\right)$ using~\cite[Theorem 4]{BSSZ12} and then push forward this intersection to $\oM_{g,n}$ and restrict it to~$\mcM_{g,n}$.

Consider the class $\pi^*\psi_1$ for $\pi\colon\oM_{g,n+k}\to\oM_{g,n}$. This class is equal to $\psi_1-D$, where~$D$ is the class of the divisor on which the points labeled by $a_1$ and $b_i$, $i\in I\subset\{1,\dots,k\}$ lie on a separate bubble. The choice of a subset $I\ne\emptyset$ here corresponds to a choice of an irreducible component of $D$. For dimensional reasons, the only non-trivial contribution under the push-forward is given by the irreducible components of $-D$ that correspond to $|I|=1$, that is, the point $a_1$ lie on a separate bubble with exactly one point~$b_i$, $i=1,\dots,k$. The push-forward $\pi_*$ of this intersection is equal to the class 
\begin{equation}\label{eq:1st-summand}
-\sum_{j=1}^k DR_g\left(m_{a_1+b_j}\prod_{i=2}^n m_{a_i}\prod_{l\ne j}\m_{b_l}\right)
\end{equation}
(See, e.~g.~the local computation of multiplicity in~\cite[Lemma 3.3]{SZ08}. Here we have a different global geometry of the space, but the local multiplicity is computed in exactly the same way).

Now we use the formula for the intersection $\psi_1\cdot DR_g\left(\prod_{i=1}^n m_{a_i}\prod_{j=1}^k m_{b_j}\right)$ in~\cite[Theorem 4]{BSSZ12}. We have there a sum over all possible degenerations of the DR-cycle into two DR-cycles
\[
DR_{g_1}\left(\prod_{i\in I} m_{a_i}\prod_{j\in J} m_{b_j}\prod_{l=1}^p m_{-c_k}\right)
\boxtimes
DR_{g_2}\left(\prod_{i\in I'} m_{a_i}\prod_{j\in J'} m_{b_j}\prod_{l=1}^p m_{c_k}\right),
\]
where $g_1+g_2+p-1=g$, $I\sqcup I'=\{1,\dots,n\}$, and $J\sqcup J'=\{1,\dots,k\}$, with some combinatorially defined coefficients. Here the symbol $\boxtimes$ refers to the gluing of the points $-c_k$ and $c_k$ to a node, $k=1,\dots,p$, and for the full definition we refer to \cite[Section 2.1]{BSSZ12}. Now we claim, that the only degenerations of the DR-cycle, that do not vanish under the restriction of the push-forward $\pi_*$ to the open moduli space $\mcM_{g,n}$, can be described in the following way:
\begin{enumerate}
\item We have $p=1$;
\item Either $g_1=0$ or $g_2=0$;
\item The genus zero component has only three special points, where one point is $m_{\pm c_1}$, and the other two are either a pair $m_{a_i}, m_{b_j}$, or a pair $m_{b_i}, m_{b_j}$.
\end{enumerate}
Indeed, if $p>1$, or both $g_1,g_2\geq 1$, or the genus $0$ component contains at least two points $a_i,a_j$, then the push-forward of this class lies on the boundary of the moduli space $\oM_{g,n}$. So, we have genus zero component with at most $1$ point $a_i$. Then this component should contain precisely two points from the list $a_1,\dots,a_n,b_1,\dots,b_k$ for dimensional reasons. This way we come to the description above.

Thus, we have three essentially different cases: the two marked points on the genus~0 component can be either $m_{b_i}, m_{b_j}$, or $m_{a_i}, m_{b_j}$, $i\not=1$, or $m_{a_1}, m_{b_j}$. In the first case, the class (including the coefficient) that we get is equal to
\[
-\frac{b_i+b_j}{ra_1}DR_g\left(\prod_{l=1}^n m_{a_l}\m_{b_i+b_j}\prod_{l\ne i,j}\m_{b_l}\right).
\]
The sum of these classes is exactly the first summand on the right hand side of equation~\eqref{eq:int-psi}. In the second case, after the push-forward and restriction to the open part, we obtain the class
\[
-\frac{a_i+b_j}{ra_1}DR_g\left(\prod_{l=1}^n m_{a_l+\delta_{l,i}b_j}\prod_{l\ne j}\m_{b_l}\right),
\]
and these classes sum up to the second summand of our final formula~\eqref{eq:int-psi}. In the last case, we have the class
\begin{equation}\label{eq:2nd-summand}
\frac{(r-1)(a_1+b_j)}{ra_1}DR_g\left(m_{a_1+b_j}\prod_{i=2}^n m_{a_i}\prod_{l\ne j}\m_{b_l}\right).
\end{equation}
Recall that we already got a contribution of the same class with a different coefficient in~\eqref{eq:1st-summand}. The sum of the classes~\eqref{eq:2nd-summand} and the corresponding summand in~\eqref{eq:1st-summand} is equal to
\[
\frac{-a_1+(r-1)b_j}{ra_1}DR_g\left(m_{a_1+b_j}\prod_{i=2}^n m_{a_i}\prod_{l\ne j}\m_{b_l}\right).
\]
The sum of these terms is precisely the third summand in equation~\eqref{eq:int-psi}.
\end{proof}

\begin{lemma}\label{lemma: running point}
Let $a_1,a_2,\ldots,a_n$, $n\ge 0$, and $b_1,b_2,\ldots,b_k$, $1\le k\le g$, be non-zero integers such that $\sum_{i=1}^n a_i+\sum_{j=1}^k b_j=0$. Then we have the following relation in~$R^{g-k+1}(\mcM_{g,n+1})$: 
\begin{gather}\label{eq:class}
\sum_{i=1}^k b_i DR_g\left(m_{b_i}\prod_{j\ne i}\m_{b_j}\prod_{l=1}^n m_{a_l}\right)\in\pi_1^*(R^{g-k+1}(\mcM_{g,n})).
\end{gather}
\end{lemma}

\begin{proof}
There are several possible proofs of this lemma. In the case $k=1$, the lemma is obvious, since the class~\eqref{eq:class} is equal to zero. If we assume that $n\geq 1$, then we can prove this lemma in the following way. We consider the class $\psi_2$ (this is the $\psi$-class at the point labelled by $a_1$) multiplied by 
$\pi_1^*DR_g\left(\prod_{j=1}^k\m_{b_j}\prod_{i=1}^n m_{a_i}\right)$. On the open part it is equal to $\pi_1^*\psi_2\cdot\pi_1^*DR_g\left(\prod_{j=1}^k\m_{b_j}\prod_{i=1}^n m_{a_i}\right)$, so it is a pull-back. On the other hand, we can compute this class using~\cite[Theorem 5]{BSSZ12} (we should assign to the point $x_1$ the multiplicity zero), and, via the same argument as in the proof of Lemma~\ref{lemma: intersection with psi-class}, we see, that the only terms, that contribute to the restriction of the class, that we obtain, to the open moduli space, are the class~\eqref{eq:class} and further classes in $\pi_1^*(R^{g-k+1}(\mcM_{g,n}))$. In the general case we use the argument very close to the one in~\cite[Proposition 2.8]{Ion02}.

Consider the space $\oM_{g,1}(\mu,\nu)$ of rubber maps to $\CP^1$ with one marked point $x_1$. We assume that the collection of multiplicities $\{\mu_1,\dots,\mu_{m_+},-\nu_1,\dots,-\nu_{m_-}\}$ is equal to $\{a_1,\dots,a_n,b_1,\dots,b_k\}$, up to a reordering. The branching morphism $\sigma$ takes the space $\oM_{g,1}(\mu,\nu)$ to the space $\mathrm{LM}_{r+1}/S_r$, that is, to the quotient of the Losev-Manin moduli space~(\cite{LM00}) by the action of the symmetric group that permutes the branch points of the rubber maps. We refer to~\cite[Section 2.2]{BSSZ12} for a full discussion of this branching morphism in this setting.

We consider the following divisors in $\mathrm{LM}_{r+1}$. Let $p_1,\dots,p_r$ be the branch points and let $q$ be the image of $x_1$. We consider the $S_r$-symmetrization of the divisor $D_{0,p_1|q,\infty}-D_{0,q|p_1,\infty}$, where by $D_{a,b|c,d}$ we denote the divisor of two-component curves, where the pairs of points $a,b$ and $c,d$ lie on different components. Denote this symmetrized divisor by $D$. The class of $D_{0,p_1|q,\infty}-D_{0,q|p_1,\infty}$ is equal to zero, therefore, the class $\sigma^*D$ of the pull-back of the symmetrization of this divisor to $\oM_{g,1}(\mu,\nu)$ is also equal to zero. 

We consider the maps $\rho\colon \oM_{g,1}(\mu,\nu)\to\oM_{g,n+k+1}$ and $\pi\colon\oM_{g,n+k+1}\to\oM_{g,n+1}$. Our goal now is to compute the restriction to the open moduli space $\mcM_{g,n+1}$ of the class~$\pi_*\rho_*\sigma^*D$. We claim that his class is equal to the sum of the class~\eqref{eq:class} and a class in $\pi_1^*(R^{g-k+1}(\mcM_{g,n}))$. Since, on the other hand, we know that  $\pi_*\rho_*\sigma^*D=0$, this proves the lemma.

So let us compute $\pi_*\rho_*\sigma^*D$. We can do it for each component separately along the lines of the computation of the analogous class in~\cite[Lemma 2.4]{BSSZ12}. Using the same argument, as in the proof of Lemma~\ref{lemma: intersection with psi-class}, we obtain a non-trivial contribution only from the divisors 
\[
DR_{g_1}\left(\prod_{i\in I} m_{a_i}\prod_{j\in J} m_{b_j}\prod_{l=1}^p m_{-c_l}\right)
\boxtimes
DR_{g_2}\left(\prod_{i\in I'} m_{a_i}\prod_{j\in J'} m_{b_j}\prod_{l=1}^p m_{c_l}\right),
\]
of two possible types. One type is exactly the divisors described in Lemma~\ref{lemma: intersection with psi-class}, that is, $p=1$, one component is of genus $0$, it contains exactly two marked points that are labeled either by $a_i,b_j$ or $b_i,b_j$, and the marked point $x_1$ lies on the other component. The last condition follows from dimensional reason. Since there are no restrictions on the position of $x_1$, all classes of this type project to $\pi_1^*(R^{g-k+1}(\mcM_{g,n}))$.

Let us describe the other possible type. We still have $p=1$ and one of the components is of genus $0$ (otherwise, the restriction to the open moduli space would be trivial), and on the component of genus $0$ we have only two marked points (that is for dimensional reason), but now these two marked points will be $x_1$ and $b_i$. This class projects to the $i$-th term in the sum~\eqref{eq:class}, and its coefficient, according to~\cite[Lemma 2.4]{BSSZ12}, is precisely~$b_i$. This completes the proof of the lemma. 
\end{proof}

\subsection{Hain's formula}\label{subsection: Hain's formula}

Consider the moduli space $\mcM^{\rm rt}_{g,n}$ of stable curves with rational tails. In this section we work in the cohomology of $\mcM^{\rm rt}_{g,n}$. Let $\psi^\dagger_i=p_i^*\psi_i$, where $p_i\colon\mcM^{\rm rt}_{g,n}\to\mcM^{\rm rt}_{g,1}$ is the forgetful map that forgets all marked points except the $i$-th. Let $J$ be any subset of $\{1,2,\ldots,n\}$ such that $|J|\ge 2$. Denote by $D_J$ the divisor in $\mcM_{g,n}^{\rm rt}$ that is formed by stable curves with a rational component that contains exactly marked points numbered by the subset~$J$.

There is the following formula discovered by Hain (\cite{Hain11}):
\begin{gather}\label{formula: Hain's formula}
DR_g\left(\prod_{i=1}^n m_{a_i}\right)=\frac{1}{g!}\left(\sum_{i=1}^n\frac{a_i^2\psi^\dagger_i}{2}-\sum_{\substack{J\subset\{1,2,\ldots,n\}\\|J|\ge 2}}\left(\sum_{i,j\in J, i< j}a_ia_j\right)D_J\right)^g.
\end{gather}
To be precise, in \cite{Hain11} this formula was proved for a version of double ramification cycles that is constructed using the universal Jacobian over the moduli space of curves. Luckily, in \cite{CMW11} it is proved that this version coincides with our version, if we restrict them to $\mcM^{\rm rt}_{g,n}$. 

From~\eqref{formula: Hain's formula} it follows that, as a cohomology class in $H^{2g}(\mcM_{g,n}^{\rm rt};\mbQ)$, the class 
$$
DR_g\left(m_{-\sum_{i=2}^n a_i}\prod_{i=2}^n m_{a_i}\right)
$$
is a homogeneous polynomial of degree $2g$ in the variables $a_2,\ldots,a_n$. Let us prove the following simple lemma.

\begin{lemma}\label{lemma: divisible}
As a cohomology class in $H^{2g-2}(\mcM_{g,n}^{\rm rt};\mbQ)$, the class 
$$
DR_g\left(m_{-\sum_{i=2}^n a_i}\prod_{i=2}^{n-1} m_{a_i}\m_{a_n}\right)
$$
is a homogeneous polynomial of degree $2g$ in the variables $a_2,\ldots,a_n$. Moreover, this polynomial is divisible by $a_n$.
\end{lemma}
\begin{proof}
Let $a_1:=-\sum_{i=2}^n a_i$. We have $DR_g\left(\prod_{i=1}^{n-1} m_{a_i}\m_{a_n}\right)=\pi_*\left(DR_g\left(\prod_{i=1}^n m_{a_i}\right)\right)$, where $\pi\colon\mcM^{\rm rt}_{g,n}\to\mcM^{\rm rt}_{g,n-1}$ is the forgetful morphism that forgets the last marked point. Thus, the first statement is clear. 

It is easy to see that
\begin{multline*}
\left.\left(\sum_{i=1}^n\frac{a_i^2\psi^\dagger_i}{2}-\sum_{\substack{J\subset\{1,2,\ldots,n\}\\|J|\ge 2}}\left(\sum_{i,j\in J, i<j}a_ia_j\right)D_J\right)\right|_{a_n=0}= \\
=\pi_n^*\left(\sum_{i=1}^{n-1}\frac{a_i^2\psi^\dagger_i}{2}-\sum_{\substack{J\subset\{1,2,\ldots,n-1\}\\|J|\ge 2}}\left(\sum_{i,j\in J, i<j}a_ia_j\right)D_J\right).
\end{multline*}
Therefore, if we set $a_n=0$ on the right-hand side of \eqref{formula: Hain's formula} and push it forward to $\mcM_{g,n-1}^{\rm rt}$, we get zero. Hence, the second statement of the lemma is also proved.
\end{proof}

\subsection{DR-cycles and the tautological ring of $\mcM_{g,n}$}\label{subsection: DR-cycles and the tautological ring}

\begin{lemma}\label{lemma: DR-cycles span the ring}
Let $n\ge 1$ and $1\le k\le g$. The group $R^{g-k}(\mcM_{g,n})$ is spanned by double ramification cycles of the form
\begin{gather}\label{special form}
DR_g\left(\prod_{i=1}^{n-1}m_{a_i}m_{-d}\prod_{j=1}^k\m_{b_j}\right),
\end{gather}
where $a_1,\ldots,a_{n-1},b_1,\ldots,b_k$ and $d$ are positive integers such that $a_1+\ldots+a_{n-1}+b_1+\ldots+b_k=d$.
\end{lemma}
\begin{proof}
This lemma is a version of~\cite[Corollary 2.5]{Ion02} adapted to our situation (we are using the DR-cycles defined by ``rubber'' stable relative maps rather than the admissible coverings, we want to have the DR-cycles with exactly one negative multiplicity, and we consider the restriction of the DR-cycles to the open part of the moduli space. 

We want to show that the push-forward of any monomial of $\psi$-classes on $\oM_{g,n+l}$ to the space $\oM_{g,n}$, and further restricted to its open part, is expressed in terms of the DR-cycles of the form~\eqref{special form}. 

We represent the degree zero class $g!\prod_{i=1}^g b_i^2$ on $\oM_{g,n+l}$ as a DR-cycle 
$$
DR_g\left(\prod_{i=1}^{n+l}m_{a_i}\prod_{j=1}^g\m_{b_j}\right),
$$
where $a_n$ is the only one negative index. Obviously, we can choose the multiplicities in this way.
Then we lift the monomial of psi-classes to the moduli space $\oM_{g,n+l+g}$ (as we did in the proof of Lemma~\ref{lemma: intersection with psi-class} above) and intersect it there with the DR-cycle $DR_g\left(\prod_{i=1}^{n+l}m_{a_i}\prod_{j=1}^g m_{b_j}\right)$ using~\cite[Theorem 4]{BSSZ12}. 
Then we apply the push-forward to the space $\oM_{g,n}$ and restrict it to the open moduli space $\mcM_{g,n}$. 

Let us discuss the structure of the class that we obtain. First of all, since we are interested only in the non-degenerate restricitons to the open moduli space, one of the components in the degeneration formula 
\[
DR_{g_1}\left(\prod_{i\in I} m_{a_i}\prod_{j\in J} m_{b_j}\prod_{l=1}^p m_{-c_k}\right)
\boxtimes
DR_{g_2}\left(\prod_{i\in I'} m_{a_i}\prod_{j\in J'} m_{b_j}\prod_{l=1}^p m_{c_k}\right),
\]
must be of genus $g$. This implies, exactly as in the proofs of Lemmas~\ref{lemma: intersection with psi-class} and~\ref{lemma: running point}, that $p=0$ and the other component has genus $0$. This means that on the component of genus~$g$ we again have the structure of a DR-cycle with exactly one negative index. So, if we are interested only in the terms that can be non-trivially restricted to the open part of the moduli space, then this property of a DR-cycle is preserved throughout the computation of the monomial of $\psi$-classes. It is also preserved under the push-forward that forgets some of the marked points. This implies that the expression that we obtain for a push-forward of a monomial of $\psi$-classes is a linear combination of DR-cycles with one negative index. It is also easy to see that the only negative index in these cycles corresponds to the $n$-th marked point. This is precisely the statement of the lemma. 
\end{proof}


\section{Technical lemmas}\label{section: technical lemmas}

Here we collect several linear algebra lemmas that we will use in Section~\ref{section: main proofs}. 

\begin{lemma}\label{lemma: linear system}
Let $V$ be a vector space and $x_1,x_2,\ldots,x_{p-2}\in V$, where $p\ge 3$. Suppose that for any positive integers $\lambda_1,\lambda_2,\lambda_3$ such that $\lambda_1+\lambda_2+\lambda_3=p$ we have
\begin{gather*}
x_{\lambda_1}+x_{\lambda_2}+x_{\lambda_3}=0.
\end{gather*}
Then there exists a vector $\alpha\in V$ such that $x_i=\left(\frac{i}{p}-\frac{1}{3}\right)\alpha$.
\end{lemma}
\begin{proof}
If $p=3$ the lemma is obvious. Suppose $p\ge 4$. We have 
$$
x_i+x_{p-1-i}=-x_1, \text{ for $i=1,2,\ldots,p-2$}.
$$
On the other hand
$$
x_i+x_{p-2-i}=-x_2, \text{ for $i=1,2,\ldots,p-3$}.
$$
Let us subtract the second equation from the first one. We obtain
$$
x_{p-1-i}-x_{p-2-i}=x_2-x_1, \text{ for $i=1,2,\ldots,p-3$}.
$$
Hence, $x_i=x_1+(i-1)(x_2-x_1)$. Let us substitute this formula in the equation $2x_1+x_{p-2}=0$, we get $x_1=-\frac{(p-3)(x_2-x_1)}{3}$. Therefore, $x_i=(i-\frac{p}{3})(x_2-x_1)$. The lemma is proved.
\end{proof}

\begin{lemma}\label{lemma: linear system2}
Let $V$ be a vector space. Suppose that we have vectors $v_{i,j}\in V, i,j\ge 1$. Define the vectors $z_{i,j}$ by
$$
z_{i,j}:=-v_{i,j}-v_{j,i}.
$$ 

Suppose that the vectors $v_{i,j}$ and $z_{i,j}$ satisfy the following system:
\begin{gather}\label{formula: first system}
\left\{
\begin{aligned}
&\sum_{\substack{\{i,j,k\}=\{1,2,3\}\\i<j}}(v_{a_k,a_i+a_j}+z_{a_i,a_j})=0,             & & \text{ if $a_1,a_2,a_3\ge 1$};\\
&z_{a_1,a_2}-z_{c_1,c_2}-\sum_{a_i>c_j}v_{c_j,a_i-c_j}+\sum_{c_i>a_j}v_{a_j,c_i-a_j}=0, & & \text{ if $a_1+a_2=c_1+c_2$}.
\end{aligned}
\right.
\end{gather}

Then there exist vectors $\alpha_2,\alpha_3,\ldots\in V$ such that 
\begin{gather}\label{formula: solution2}
v_{i,j}=\left(\frac{i}{i+j}-\frac{1}{3}\right)\alpha_{i+j}+\left(\frac{1}{3}-\frac{i+j}{i}\right)\alpha_i+\frac{1}{3}\alpha_j.
\end{gather}
Here we, by definition, put $\alpha_1:=0$.
\end{lemma}
\begin{proof}
Obviously, $\alpha_2=6v_{1,1}$. Suppose $d\ge 3$ and we have found vectors $\alpha_2,\ldots,\alpha_{d-1}$ such that formula \eqref{formula: solution2} holds for $i+j\le d-1$. Let us construct a vector $\alpha_d$ such that equation~\eqref{formula: solution2} holds for $i+j=d$. For $i+j=d$, define the vectors $\widetilde v_{i,j}$ and $\widetilde z_{i,j}$ by
\begin{align*}
&\widetilde v_{i,j}:=v_{i,j}-\left(\frac{1}{3}-\frac{d}{i}\right)\alpha_i-\frac{1}{3}\alpha_j,\\
&\widetilde z_{i,j}:=-\widetilde v_{i,j}-\widetilde v_{j,i}.
\end{align*} 
It is easy to check that from \eqref{formula: first system} it follows that
\begin{gather*}
\left\{
\begin{aligned}
&\sum_{\substack{\{i,j,k\}=\{1,2,3\}\\i<j}}\widetilde v_{a_k,a_i+a_j}=0,             & & \text{ if $a_1+a_2+a_3=d$};\\
&\widetilde z_{a_1,a_2}-\widetilde z_{c_1,c_2}=0, & & \text{ if $a_1+a_2=c_1+c_2=d$}.
\end{aligned}
\right.
\end{gather*}
If $d=3$, then we can take $\alpha_3=3\widetilde v_{2,1}$. Suppose $d\ge 4$. By Lemma~\ref{lemma: linear system}, there exists a vector $\alpha_d$ such that $\widetilde v_{i,d-i}=\left(\frac{i}{d}-\frac{1}{3}\right)\alpha_d$, for $i\le d-2$. Since $\widetilde z_{1,d-1}=\widetilde z_{2,d-2}$, we also have $\widetilde v_{d-1,1}=\left(\frac{d-1}{d}-\frac{1}{3}\right)\alpha_d$. The lemma is proved. 
\end{proof}

\begin{lemma}\label{lemma: linear system3}

Let $V$ be a vector space and $d\ge 3$. Suppose we have vectors $v_{i,j,k}\in V$, where $i+j+k=d$ and $i,j,k\ge 1$, that satisfy the following system of equations:
\begin{align}
&\sum_{\substack{\{i,j,k\}=\{1,2,3\}\\i<j}}v_{b_k,b_i+b_j,a}=0,\text{ if $b_1+b_2+b_3+a=d$ and $a,b_i\ge 1$};\label{first equation}\\
&k v_{i,j,k}+j v_{i,k,j}=0.\label{second equation}
\end{align}

If $d=3,4$, then $v_{i,j,k}=0$. If $d\ge 5$, then there exists a vector $\alpha\in V$ such that
$$
v_{i,j,k}=\left(\frac{i}{d-1}-\frac{1}{3}\right)\left(\delta_{k,1}-\frac{1}{k}\delta_{j,1}\right)\alpha.
$$
\end{lemma}
\begin{proof}
The cases $d=3,4,5$ can be easily done by a direct computation. 

Suppose that $d\ge 6$. By Lemma \ref{lemma: linear system}, there exist vectors $\alpha_3,\alpha_4,\ldots,\alpha_{d-1}\in V$ such that
\begin{gather*}
v_{i,j,k}=\left(\frac{i}{i+j}-\frac{1}{3}\right)\alpha_{i+j}, \text{ if $j\ge 2$}.
\end{gather*}

Let us prove that 
\begin{gather}\label{formula: tmp3}
\alpha_{d-k}=0,\text{ if $2\le k\le\frac{d-1}{2}$}.
\end{gather}
From \eqref{second equation} it follows that $v_{i,j,k}=0$, if $j=k$. Therefore, if $2\le k\le\frac{d-1}{2}$, then $\left(\frac{d-2k}{d-k}-\frac{1}{3}\right)\alpha_{d-k}=0$. If $d\ne \frac{5}{2}k$, then $\alpha_{d-k}=0$. Suppose that $d=\frac{5}{2}k$. We have $v_{\frac{k}{2}+1,k-1,k}=-\frac{k-1}{k}v_{\frac{k}{2}+1,k,k-1}=0$. On the other hand, $v_{\frac{k}{2}+1,k-1,k}=\left(\frac{\frac{k}{2}+1}{d-k}-\frac{1}{3}\right)\alpha_{d-k}$. The coefficient of $\alpha_{d-k}$ here is not equal to zero, therefore, $\alpha_{d-k}=0$. Thus, \eqref{formula: tmp3} is proved.

We see that $v_{i,j,k}=0$, if $j\ge 2$ and $2\le k\le\frac{d-1}{2}$. If we apply \eqref{second equation}, we get that $v_{i,j,k}=0$, if $k\ge 2$ and $2\le j\le\frac{d-1}{2}$, thus, $v_{i,j,k}=0$, if $j\ge 2$ or $k\ge 2$.

If $2\le j\le d-2$, then $v_{d-j-1,j,1}=\left(\frac{d-j-1}{d-1}-\frac{1}{3}\right)\alpha_{d-1}$ and $v_{d-j-1,1,j}=-\frac{1}{j}v_{d-j-1,1,j}=-\frac{1}{j}\left(\frac{d-j-1}{d-1}-\frac{1}{3}\right)\alpha_{d-1}$. Also we have $v_{d-2,1,1}=0$. This completes the proof of the lemma.
\end{proof}


\section{Proofs of Lemma~\ref{lemma: symmetry} and Propositions~\ref{proposition: first step} and~\ref{proposition: second step}}\label{section: main proofs}

In this last section we prove Lemma~\ref{lemma: symmetry}, Proposition~\ref{proposition: first step} and Proposition~\ref{proposition: second step}.

\subsection{Proof of Lemma \ref{lemma: symmetry}}\label{subsection: symmetry}

Without loss of generality, we can assume that $k=1$ and $l=2$. It is sufficient to prove that
\begin{gather}\label{formula1}
\alpha+i_{1,2}^*\alpha\in\pi_1^*(R^{g-1}(\mcM_{g,n-1}))+\pi_2^*(R^{g-1}(\mcM_{g,n-1})),
\end{gather} 
for any $\alpha\in R^{g-1}(\mcM_{g,n})$. By Lemma~\ref{lemma: DR-cycles span the ring}, we can assume that $\alpha=DR_g\left(\prod_{i=1}^n m_{a_i}\m_b\right)$. From Lemma \ref{lemma: running point} it follows that 
\begin{align*}
&DR_g\left(m_{a_1}\m_b m_{a_2}\prod_{i=3}^n m_{a_i}\right)+\frac{b}{a_1}DR_g\left(m_b\m_{a_1}m_{a_2}\prod_{i=3}^n m_{a_i}\right)\in\pi_1^*(R^{g-1}(\mcM_{g,n-1})),\\
&\frac{b}{a_1}DR_g\left(m_b m_{a_2}\m_{a_1}\prod_{i=3}^n m_{a_i}\right)+\frac{b}{a_2}DR_g\left(m_b m_{a_1}\m_{a_2}\prod_{i=3}^n m_{a_i}\right)\in\pi_2^*(R^{g-1}(\mcM_{g,n-1})),\\
&\frac{b}{a_2}DR_g\left(m_b\m_{a_2}m_{a_1}\prod_{i=3}^n m_{a_i}\right)+DR_g\left(m_{a_2}\m_b m_{a_1}\prod_{i=3}^n m_{a_i}\right)\in\pi_1^*(R^{g-1}(\mcM_{g,n-1})).
\end{align*}
If we sum the first and the third rows and subtract the second row, we get
\begin{multline*}
DR_g\left(m_{a_1}m_{a_2}\m_b\prod_{i=3}^n m_{a_i}\right)+DR_g\left(m_{a_2} m_{a_1}\m_b\prod_{i=3}^n m_{a_i}\right)\in\\
\in\pi_1^*(R^{g-1}(\mcM_{g,n-1}))+\pi_2^*(R^{g-1}(\mcM_{g,n-1})).
\end{multline*}
The lemma is proved.


\subsection{Proof of Proposition~\ref{proposition: first step}}\label{subsection: first step}

Suppose $n=1$. We know that $R^{g-1}(\mcM_{g,1})=\mbQ$, $R^{g-2}(\mcM_g)=\mbQ$ and that $R^{g-2}(\mcM_g)$ is spanned by $\kappa_{g-2}$ (see \cite{Loo95}). Therefore, it is sufficient to prove that $\psi_1\pi_1^*(\kappa_{g-2})\ne 0$. It is true, because $(\pi_1)_*(\psi_1\pi_1^*(\kappa_{g-2}))=(2g-2)\kappa_{g-2}$. 

Suppose that $n\ge 2$. We denote by $K$ the subspace of $R^{g-1}(\mcM_{g,n})$ defined by 
$$
K:=\pi_1^*(R^{g-1}(\mcM_{g,n-1}))+\psi_1\pi_1^*(R^{g-2}(\mcM_{g,n-1})).
$$
By Lemma~\ref{lemma: DR-cycles span the ring}, it is sufficient to prove that
\begin{gather}\label{formula: first step}
DR_g\left(\m_b\prod_{i=1}^{n-1}m_{a_i}m_{-d}\right)\in K,
\end{gather}
where $a_1,\ldots,a_{n-1},b,d$ are positive integers such that $a_1+\ldots+a_{n-1}+b=d$. We do it by double induction on $d$ and on $\td:=b+a_1$. If $\td=2$, then $b=a_1=1$, and from Lemma~\ref{lemma: running point} it immediately follows that
\begin{gather}\label{formula: base}
DR_g\left(\m_1 m_1\prod_{i=2}^n m_{a_n}m_{-d}\right)\in\pi_1^*(R^{g-1}(\mcM_{g,n-1})).
\end{gather}

Suppose that $\td\ge 3$. Consider positive integers $b_1,b_2,b_3$ such that $b_1+b_2+b_3=\td$. By Lemma~\ref{lemma: running point}, we have
$$
\sum_{i=1}^3 b_i DR_g\left(m_{b_i}\prod_{j\ne i}\m_{b_j}\prod_{l=2}^{n-1} m_{a_l}m_{-d}\right)\in\pi_1^*(R^{g-2}(\mcM_{g,n-1})).
$$
Let us multiply the both sides of this formula by $\psi_1$. Using Lemmas \ref{lemma: intersection with psi-class}, \ref{lemma: running point} and the induction assumption we get
\begin{gather}\label{formula: relation1}
\sum_{i=1}^3 \frac{\td+(r-3)b_i}{r}DR_g\left(\m_{\td-b_i}m_{b_i}\prod_{j=2}^{n-1}m_{a_j}m_{-d}\right)\in K,
\end{gather}
where $r=2g+n$. 

Let us fix $a_2,a_3,\ldots,a_n$ and analyze relation~\eqref{formula: relation1}. Let 
$$
u_i:=DR_g\left(\m_{\td-i}m_i\prod_{j=2}^{n-1}m_{a_j}m_{-d}\right), \text{ for $i=1,2,\ldots,\td-1$}.
$$
For any two classes $\alpha,\beta\in R^{g-1}(\mcM_{g,n})$, we will write $\alpha\stackrel{\mod K}{=}\beta$, if $\alpha-\beta\in K$. 

From \eqref{formula: relation1} and Lemma \ref{lemma: linear system} it follows that there exists a class $\alpha\in R^{g-1}(\mcM_{g,n})$ such that
\begin{gather}\label{equation1}
\frac{\td+(r-3)i}{r}u_i\stackrel{\mod K}{=}\left(\frac{i}{\td}-\frac{1}{3}\right)\alpha, \text{ for $i=1,\ldots,\td-2$}.
\end{gather}
By Lemma \ref{lemma: running point}, we have
\begin{gather}\label{equation2}
iu_i+(\td-i)u_{\td-i}\in\pi_1^*(R^{g-1}(\mcM_{g,n-1})), \text{ for $i=1,2,\ldots,\td-1$}.
\end{gather}
Let us prove that relations \eqref{equation1} and \eqref{equation2} imply that $u_i\in K$, for $i=1,\ldots,\td-1$. This will complete the proof of the proposition.

Suppose $\td=3$, then from \eqref{equation1} it follows that $u_1\in K$, and from \eqref{equation2} it follows that~$u_2\in K$.

Suppose $\td\ge 4$. From \eqref{equation1} and \eqref{equation2} we see that it is sufficient to prove that $\alpha\in K$. Let $2\le i\le\td-2$. From \eqref{equation1} it follows that
\begin{gather}
iu_i+(\td-i)u_{\td-i}\stackrel{\mod K}{=}
\frac{\left(i-\frac{\td}{2}\right)^2(3-\frac{r}{3})+\frac{\td^2(r-1)}{12}}{\left(\td+(r-3)i\right)\left(\td+(r-3)(\td-i)\right)}r\alpha.\label{equation3}
\end{gather}

If $\td=4$ or $\td=5$, then, for $i=2$, the numerator in \eqref{equation3} is nonzero. Therefore,~$\alpha\in K$. Suppose $\td\ge 6$. It is clear that the numerator in \eqref{equation3} cannot be zero for two different values of $i$ that are greater or equal to $\frac{\td}{2}$. Hence, $\alpha\in K$. The proposition is proved.


\subsection{Proof of Proposition \ref{proposition: second step}: main relation}\label{subsection: second step}

In this section we construct a relation between double ramification cycles that is a main ingredient in the proof of Proposition~\ref{proposition: second step}. 

In Section~\ref{subsubsection: basic relation} we construct a basic relation using the same idea as in the proof of Proposition~\ref{proposition: first step}. The problem is that this relation is too complicated to work with. In Section~\ref{subsubsection: new variables} we introduce new variables that are linear combinations of double ramification cycles. This change of variables allows us to simplify the basic relation a lot. This is done in Section~\ref{subsubsection: main relation}.

\subsubsection{Basic relation}\label{subsubsection: basic relation}
  
We denote by $K$ the subspace of $R^{g-2}(\mcM_{g,n})$ defined by
$$
K=\pi_1^*(R^{g-2}(\mcM_{g,n-1}))+\psi_1\pi_1^*(R^{g-3}(\mcM_{g,n-1}))+\sum_{1\le i<j\le n}S^{g-2}_{i,j}(\mcM_{g,n}).
$$ 

Let us fix a triple $((b_1,b_2,b_3),b_4,(a_1,a_2,\ldots,a_{n-1}))$, where $(b_1,b_2,b_3)$ is an unordered triple of non-zero integers, $b_4$ is a non-zero integer, $(a_1,a_2,\ldots,a_{n-1})$ is an ordered sequence of non-zero integers and $\sum b_i+\sum a_j=0$. By Lemma \ref{lemma: running point}, we have
\begin{gather}\label{formula: running point relation}
\sum_{i=1}^4 b_i DR_g\left(m_{b_i}\prod_{j\ne i}\tm_{b_j}\prod_{p=1}^{n-1} m_{a_p}\right)\in\pi_1^*(R^{g-3}(\mcM_{g,n-1})).
\end{gather}
Let us multiply both sides of~\eqref{formula: running point relation} by $\psi_1$. Using Lemmas \ref{lemma: intersection with psi-class} and~\ref{lemma: running point} we get
\begin{align}
&\sum_{\substack{\{i,j,k\}=\{1,2,3\}\\i<j}}\frac{b_i+b_j}{r}\left(r-2+\frac{b_i+b_j}{b_4}\right)DR_g\left(\tm_{b_4}\tm_{b_k}m_{b_i+b_j}\prod_{l=1}^{n-1} m_{a_l}\right)\label{formula: newrelation1}\\
&-\sum_{\substack{\{i,j,k\}=\{1,2,3\}\\i<j}}\frac{b_i+b_j}{r}\left(1-\frac{b_k}{b_4}\right)DR_g\left(\tm_{b_4}\tm_{b_i+b_j}m_{b_k}\prod_{l=1}^{n-1} m_{a_l}\right)\notag\\
&-\sum_{l=1}^{n-1}\sum_{\{i,j,k\}=\{1,2,3\}}\frac{b_i+a_l}{r}\left(1-\frac{b_j}{b_4}\right)DR_g\left(\tm_{b_4}\tm_{b_k}m_{b_j}\prod_{p=1}^{n-1} m_{a_p+\delta_{p,l}b_i}\right)\notag\\
&+\sum_{l=1}^{n-1}\sum_{\substack{\{i,j,k\}=\{1,2,3\}\\i<j}}\frac{b_j-b_i}{r}DR_g\left(\tm_{b_4+a_l}\tm_{b_k}m_{b_i}\prod_{p=1}^{n-1} m_{a_p+\delta_{p,l}(b_j-a_l)}\right)\notag\\
&-\sum_{\{i,j,k\}=\{1,2,3\}}\frac{b_4+b_k+(r-2)b_i}{r}DR_g\left(\tm_{b_4+b_k}\tm_{b_j}m_{b_i}\prod_{l=1}^{n-1} m_{a_l}\right)\in K,\notag
\end{align}
where $r=2g+n+1$. This relation will be called the basic relation. 

\subsubsection{New variables}\label{subsubsection: new variables}

Let $f_1,f_2,\ldots,f_{n+1}$ be arbitrary integers such that not all of them are equal to zero. We introduce the cycle $Z_g\left(\prod_{i=1}^{n+1}m_{f_i}\right)\in R^{g-2}(\mcM_{g,n})$ as follows:
\begin{align*}
Z_g\left(\prod_{i=1}^{n+1}m_{f_i}\right):=&\frac{f_1-(r-n-2)f_2}{r}DR_g\left(\tm_{-d}\tm_{f_1}m_{f_2}\prod_{l=3}^{n+1}m_{f_l}\right)\\
&+\frac{f_2-(r-n-2)f_1}{r}DR_g\left(\tm_{-d}\tm_{f_2}m_{f_1}\prod_{l=3}^{n+1}m_{f_l}\right)\\
&+\frac{f_2-f_1}{r}\sum_{l=3}^{n+1}DR_g\left(\tm_{-d}\tm_{f_l}m_{f_2}\prod_{p=3}^{n+1}m_{f_p+\delta_{p,l}(f_1-f_p)}\right),
\end{align*}
where $r:=2g+n+1$ and $d:=\sum_{i=1}^{n+1}f_i$. Suppose that $d\ne 0$. Define the cycle~$V_g\left(\prod_{i=1}^{n+1}m_{f_i}\right)\in R^{g-2}(\mcM_{g,n})$ by
\begin{align*}
V_g\left(\prod_{i=1}^{n+1}m_{f_i}\right):=&\frac{f_2}{r}\left[\left(r-n-2+\frac{f_1}{d}\right)DR_g\left(\tm_{-d}\tm_{f_1}\prod_{l=2}^{n+1}m_{f_l}\right)\right.\\
&-\left.\left(1+\frac{f_1}{d}\right)\sum_{l=2}^{n+1}DR_g\left(\tm_{-d}\tm_{f_l}\prod_{p=2}^{n+1}m_{f_p+\delta_{p,l}(f_1-f_p)}\right)\right].
\end{align*}
Note that the cycle $Z_g$ is a linear combination of double ramification cycles of the same degree. Thus, the degree of the cycle $Z_g$ is well defined. The same is true for the cycle $V_g$.   

From the definition it immediately follows that
\begin{gather}\label{eq:zero V_g}
V_g\left(\prod_{i=1}^{n+1}m_{f_i}\right)=0,\quad\text{if $f_2=0$}.
\end{gather}
Lemma~\ref{lemma: divisible} implies that
\begin{gather}\label{eq:zero Z_g}
Z_g\left(\prod_{i=1}^{n+1}m_{f_i}\right)=0,\quad\text{if $\sum f_i=0$}.
\end{gather}

It is not hard to show that
$$
V_g\left(m_{f_1}m_{f_2}\prod_{i=3}^{n+1}m_{f_i}\right)+V_g\left(m_{f_2}m_{f_1}\prod_{i=3}^{n+1}m_{f_i}\right)\stackrel{\mod K}{=}-Z_g\left(\prod_{i=1}^{n+1}m_{f_i}\right).
$$
Using Lemma~\ref{lemma: running point} we can easily derive the following relations:
\begin{align}
&f_iV_g\left(\prod_{i=1}^{n+1}m_{f_i}\right)+f_2V_g\left(m_{f_1}m_{f_i}\prod_{k=3}^{n+1}m_{f_{k+\delta_{i,k}(2-k)}}\right)\in K,&&\text{if $3\le i\le n+1$};\label{formula: 2,i-symmetry}\\
&V_g\left(\prod_{i=1}^{n+1}m_{f_i}\right)+V_g\left(m_{f_1}m_{f_2}\prod_{k=3}^{n+1}m_{f_{k+\delta_{i,k}(j-k)+\delta_{j,k}(i-k)}}\right)\in K,&&\text{if $3\le i<j\le n+1$};\label{formula: i,j-symmetry}\\
&V_g\left(\prod_{i=1}^{n+1}m_{f_i}\right)\stackrel{\mod K}{=}V_g\left(m_{-d}\prod_{i=2}^{n+1}m_{f_i}\right),&&\text{if $f_1\ne 0$};\label{formula: 1-symmetry}\\
&V_g\left(\prod_{i=1}^{n+1}m_{f_i}\right)\stackrel{\mod K}{=}Z_g\left(m_{f_1}m_{-d}\prod_{i=3}^{n+1}m_{f_i}\right).&&\label{formula: 2-symmetry}
\end{align}
Also from~\eqref{formula: change sign in DR-cycle} it follows that
\begin{gather}\label{formula: change sign in V-cycle}
V_g\left(\prod_{i=1}^{n+1}m_{f_i}\right)=-V_g\left(\prod_{i=1}^{n+1}m_{-f_i}\right).
\end{gather}

\begin{lemma}\label{lemma: span}
The space $R^{g-2}(\mcM_{g,n})/K$ is spanned by cycles $V_g\left(\prod_{i=1}^{n+1}m_{f_i}\right)$ with positive~$f_i$'s.
\end{lemma}
\begin{proof}
Lemmas~\ref{lemma: running point} and~\ref{lemma: DR-cycles span the ring} imply that the space $R^{g-2}(\mcM_{g,n})/K$ is spanned by double ramification cycles of the form
$$
DR_g\left(\tm_{-d}\tm_b\prod_{i=1}^n m_{a_i}\right),
$$
where $a_1,\ldots,a_n$ and $b$ are positive integers and $d=b+\sum a_i$. For $1\le k\le n+1$, let 
\begin{align*}
&\alpha_k:=DR_g\left(\tm_{-d}\tm_{f_k}\prod_{p=2}^{n+1}m_{f_p+\delta_{p,k}(f_1-f_p)}\right),\\
&\beta_k:=V_g\left(\prod_{i=k}^{n+1}m_{f_i}\prod_{j=1}^{k-1} m_{f_j}\right).
\end{align*}
From the definition of cycles $V_g$ it is easy to compute that $\beta_i\stackrel{\mod K}{=}\sum_{j=1}^{n+1}g_{i,j}\alpha_j$, where
\begin{gather*}
g_{i,j}=
\begin{cases}
(-1)^{(i-1)n+1+\delta_{j,1}}\frac{f_{i+1}}{r}\left(r-n-2+\frac{f_i}{d}\right),&\text{if $i=j$};\\
(-1)^{(i-1)n+1+\delta_{j,1}}\frac{f_{i+1}}{r}\left(1+\frac{f_i}{d}\right),&\text{if $i\ne j$}.
\end{cases}
\end{gather*}
Here we, by definition, put $f_{n+2}:=f_1$. We see that it is sufficient to prove that the matrix $G:=(g_{i,j})$ is non-degenerate. Consider the matrix $\widetilde G:=(\tilde g_{i,j})$ defined by 
\begin{gather*}
\tilde g_{i,j}=
\begin{cases}
r-n-2+\frac{f_i}{d},&\text{if $i=j$};\\
1+\frac{f_i}{d},&\text{if $i\ne j$}.
\end{cases}
\end{gather*}
Denote by $D$ the diagonal matrix with the entries $(-1)^{(i-1)n+1}\frac{f_{i+1}}{r}$ on the diagonal. It is clear that the matrices $G$ and $\widetilde G$ are related by 
$$
G=D\cdot\widetilde G\cdot\diag(-1,1,1\ldots,1).
$$
Therefore, it is sufficient to prove that the matrix $\widetilde G$ is non-degenerate. It is easy to compute that $\det\widetilde G=(r-n-3)^n(r-1)\ne 0$, so the lemma is proved.
\end{proof}

\subsubsection{Main relation}\label{subsubsection: main relation}

Let us consider the same triple $((b_1,b_2,b_3),b_4,(a_1,a_2,\ldots,a_{n-1}))$, as in Section~\ref{subsubsection: basic relation}. Choose an arbitrary $1\le p\le n-1$ and consider four basic relations corresponding to the following triples:
\begin{align*}
&((b_1,b_2,b_3),b_4,(a_1,a_2,\ldots,a_{n-1})),\\
&((b_1,b_2,a_p),b_4,(a_1,\ldots,a_{p-1},b_3,a_{p+1},\ldots,a_{n-1})),\\
&((b_1,a_p,b_3),b_4,(a_1,\ldots,a_{p-1},b_2,a_{p+1},\ldots,a_{n-1})),\\
&((a_p,b_2,a_3),b_4,(a_1,\ldots,a_{p-1},b_1,a_{p+1},\ldots,a_{n-1})).
\end{align*}
Let us sum these relations with the coefficients $\left(1-\frac{a_p}{b_4}\right)$, $\left(1-\frac{b_3}{b_4}\right)$, $\left(1-\frac{b_2}{b_4}\right)$ and $\left(1-\frac{b_1}{b_4}\right)$ correspondingly. We get the following:
\begin{multline}\label{formula: newrelation2}
\sum_{\substack{\{i,j,k,l\}=\{1,2,3,4\}\\i<j}}\left[(c_i+c_j)\left(1-\frac{c_k}{b_4}\right)DR_g\left(\tm_{b_4}\tm_{c_l}m_{c_i+c_j}\prod_{q=1}^{n-1}m_{a_q+\delta_{q,p}(c_k-a_q)}\right)\right.\\
+\left.(c_j-c_i)DR_g\left(\tm_{b_4+c_l}\tm_{c_k}m_{c_i}\prod_{q=1}^{n-1}m_{a_q+\delta_{q,p}(c_j-a_q)}\right)\right]\in K, 
\end{multline}
where $c_i=b_i$, for $1\le i\le 3$, and $c_4=a_p$. Consider the sum of relations~\eqref{formula: newrelation2}, for $p=1,2,\ldots,n-1$. If we subtract this sum, divided by $r$, from~\eqref{formula: newrelation1}, we get
\begin{gather}\label{formula: main relation}
\sum_{\substack{\{i,j,k\}=\{1,2,3\}\\i<j}}\left[V_g\left(m_{b_k}m_{b_i+b_j}\prod_{l=1}^{n-1}m_{a_l}\right)+Z_g\left(m_{b_i}m_{b_j}\prod_{l=1}^{n-1}m_{a_l}\right)\right]\in K.
\end{gather}
This relation will be called the main relation.

\begin{remark}
In the subsequent sections we want to work with relations between cycles $V_g\left(\prod_{i=1}^{n+1}m_{f_i}\right)$ and $Z_g\left(\prod_{i=1}^{n+1}m_{f_i}\right)$, where $f_i\ne 0$ and $\sum f_i\ne 0$. At first glance, relation~\eqref{formula: main relation} contains the other cycles as well, because in the first summand the multiplicity~$b_i+b_j$ can be equal to zero, and in the second summand the sum of the multiplicities $b_i+b_j+\sum a_l=-b_4-b_k$ can be equal to zero. Happily, these extra terms vanish due to properties~\eqref{eq:zero V_g} and~\eqref{eq:zero Z_g}.
\end{remark}

\begin{remark}
In the case $n\ge 2$ our proof of Proposition~\ref{proposition: second step} is based only on the main relation~\eqref{formula: main relation} and relations~\eqref{formula: 2,i-symmetry}-\eqref{formula: change sign in V-cycle}. The case $n=1$ is exceptional, because we also have to use Lemma~\ref{lemma: divisible}.
\end{remark}

\subsubsection{Proof of Proposition \ref{proposition: second step}: the case $n=1$}\label{subsection: n=1}

Let $v_{i,j}:=V_g(m_im_j)$ and $z_{i,j}:=Z_g(m_im_j)$. By Lemma~\ref{lemma: span}, it is sufficient to prove that~$v_{i,j}\in K$, for $i,j\ge 1$. 

First of all let us write relations~\eqref{formula: 1-symmetry},~\eqref{formula: 2-symmetry} and \eqref{formula: change sign in V-cycle} in this case:
\begin{gather}\label{eq:tmp n=1 relations}
v_{i,j}\stackrel{\mod K}{=}v_{-(i+j),j},\qquad v_{i,j}\stackrel{\mod K}{=}z_{i,-(i+j)},\qquad v_{i,j}\stackrel{\mod K}{=}-v_{-i,-j}.
\end{gather}

Let $a_1,a_2,a_3$ be positive integers and $d=a_1+a_2+a_3$.  Let us write relation~\eqref{formula: main relation} for~$b_1=a_1, b_2=a_2, b_3=a_3$ and $b_4=-d$:
\begin{gather}\label{formula: n=1, relation1}
\sum_{i=1}^3 v_{a_i,d-a_i}+\sum_{i<j}z_{a_i,a_j}\in K.
\end{gather}

Let $a_1,a_2$ and $c_1,c_2$ be positive integers such that $a_1+a_2=c_1+c_2$. Let us write relation~\eqref{formula: main relation} for $b_1=a_1$, $b_2=a_2$, $b_3=-c_1$ and $b_4=-c_2$. We get a linear combination of cycles~$z_{i,j}$ and~$v_{k,l}$, where indices $i,j,k,l$ can be negative. If we apply relations~\eqref{eq:tmp n=1 relations} in order to make all indices positive, we get
\begin{gather}\label{formula: n=1, relation2}
z_{a_1,a_2}-z_{c_1,c_2}-\sum_{a_i>c_j}v_{c_j,a_i-c_j}+\sum_{c_i>a_j}v_{a_j,c_i-a_j}\in K.
\end{gather}

From Lemma~\ref{lemma: linear system2} it follows that, for any $d\ge 2$, there is a class $\alpha_d\in R^{g-2}(\mcM_{g,1})$ such that 
\begin{gather}\label{formula: solution}
v_{i,j}\stackrel{\mod K}{=}\left(\frac{i}{i+j}-\frac{1}{3}\right)\alpha_{i+j}+\left(\frac{1}{3}-\frac{i+j}{i}\right)\alpha_i+\frac{1}{3}\alpha_j.
\end{gather}
Here we, by definition, put $\alpha_1:=0$.

From \eqref{formula: Hain's formula} it follows that $v_{ai,aj}=a^{2g+1}v_{i,j}$. Hence, $v_{d,d}=d^{2g+1}v_{1,1}$. Using \eqref{formula: solution} we obtain
\begin{gather}\label{double}
\alpha_{2d}\stackrel{\mod K}{=}8\alpha_d+d^{2g+1}\alpha_2.
\end{gather}
We have $v_{2d-2,2}=2^{2g+1}v_{d-1,1}$. Using \eqref{formula: solution} we get
\begin{multline*}
\left(\frac{d-1}{d}-\frac{1}{3}\right)\alpha_{2d}+\left(\frac{1}{3}-\frac{d}{d-1}\right)\alpha_{2d-2}+\frac{\alpha_2}{3}\\
\stackrel{\mod K}{=}2^{2g+1}\left(\left(\frac{d-1}{d}-\frac{1}{3}\right)\alpha_{d}+\left(\frac{1}{3}-\frac{d}{d-1}\right)\alpha_{d-1}\right).
\end{multline*}
If we combine this equation with \eqref{double}, we get
\begin{multline}\label{recursion}
\left(\frac{d-1}{d}-\frac{1}{3}\right)(8-2^{2g+1})\alpha_d+\left(\frac{1}{3}-\frac{d}{d-1}\right)(8-2^{2g+1})\alpha_{d-1}+\\
+\left(\left(\frac{d-1}{d}-\frac{1}{3}\right)d^{2g+1}+\left(\frac{1}{3}-\frac{d}{d-1}\right)(d-1)^{2g+1}+\frac{1}{3}\right)\alpha_2\in K.\end{multline}

Relation \eqref{recursion} allows to compute all $\alpha_d$, for $d\ge 3$, it terms of $\alpha_2$. It is not hard to check that this recursion has the following solution:
$$
\alpha_d\stackrel{\mod K}{=}\frac{d^3-d^{2g+1}}{2^3-2^{2g+1}}\alpha_2.
$$
Therefore, we have
\begin{gather}\label{formula: formula for v}
v_{i,j}\stackrel{\mod K}{=}\left[\left(\frac{i}{i+j}-\frac{1}{3}\right)(i+j)^{2g+1}+\left(\frac{1}{3}-\frac{i+j}{i}\right)i^{2g+1}+\frac{j^{2g+1}}{3}\right]\alpha_2.
\end{gather}

From Lemma~\ref{lemma: divisible} it follows that the class $DR_g(\m_a\m_b m_{-a-b})$ is a homogeneous polynomial of degree $2g$ in the variables $a$ and $b$. Moreover, this polynomial is divisible by $ab$. Hence, we have
\begin{gather}\label{formula: limit}
\lim_{a\to\infty}\frac{1}{a^{2g}}DR_g(\m_a\m_1 m_{-a-1})=0.
\end{gather}
It is easy to compute that
$$
\frac{(r-1)(r-4)}{r-3}DR_g(\m_i\m_j m_{-i-j})\stackrel{\mod K}{=}\frac{r}{r-3}\frac{iv_{i,j}+jv_{j,i}}{ij}+\frac{r}{i+j}(v_{i,j}+v_{j,i}).
$$
If we substitute here the expression~\eqref{formula: formula for v} for $v_{i,j}$, we get
$$
\lim_{a\to\infty}\frac{1}{a^{2g}}DR_g(\m_a\m_1 m_{-a-1})\stackrel{\mod K}{=}\frac{r}{(r-1)(r-4)}\frac{4g-4}{3}\alpha_2.
$$
From~\eqref{formula: limit} it now follows that $\alpha_2\in K$. This completes the proof of the proposition in the case $n=1$.

\subsubsection{Proof of Proposition \ref{proposition: second step}: the case $n=2$}\label{subsubsection: n=2}

Let $v_{i,j,k}:=V_g(m_im_km_k)$ and $z_{i,j,k}:=Z_g(m_im_km_k)$. By Lemma~\ref{lemma: span}, it is sufficient to prove that $v_{i,j,k}\in K$ for $i,j,k\ge 1$. The problem here is that relations~\eqref{formula: main relation}, that involve only these terms, are not enough. So we have to consider also classes $v_{i,j,k}$ with negative indices.

In this section we consider cycles $v_{i,j,k}$ and $z_{i,j,k}$ as elements of $R^{g-2}(\mcM_{g,n})/K$. So, instead of writing $v_{i,j,k}\stackrel{\mod K}{=}0$, we will simply write $v_{i,j,k}=0$.

Let us write relations \eqref{formula: 2,i-symmetry}, \eqref{formula: 1-symmetry}, \eqref{formula: 2-symmetry} and \eqref{formula: change sign in V-cycle} in this case:
\begin{align}
&k v_{i,j,k}+j v_{i,k,j}=0,\label{formula: a1}\\
&v_{i,j,k}=v_{-(i+j+k),j,k},\label{formula: a2}\\
&v_{i,j,k}=z_{i,-(i+j+k),k},\label{formula: a3}\\
&v_{i,j,k}=-v_{-i,-j,-k}.\label{formula: a4}
\end{align}
From these relations it follows that any class $v_{a,b,c}$, where $a,b,c\ne 0$ and $a+b+c\ne 0$, can be expressed in terms of classes~$v_{i,j,k}$ with $i,j,k\ge 1$ or classes $v_{i,j,-k}$ with $i,j,k\ge 1$ and $k<i+j$. Therefore, we have to prove that  
\begin{align}
&v_{i,j,k}=0, \text{ where $i,j,k\ge 1$},\label{formula: first}\\
&v_{i,j,-k}=0, \text{ where $i,j,k\ge 1$ and $k<i+j$}.\label{formula: second}
\end{align}

We prove it by induction on the degree $d$. In~\eqref{formula: first} the degree of $v_{i,j,k}$ is $i+j+k$ and in~\eqref{formula: second} the degree of $v_{i,j,-k}$ is $i+j$. The smallest possible degree is $2$ and the class $v_{1,1,-1}$ if the only class of degree $2$. We have $v_{1,1,-1}\mathop{=}\limits^{\text{by \eqref{formula: a1}}}v_{1,-1,1}\mathop{=}\limits^{\text{by \eqref{formula: a2}}}v_{-1,-1,1}\mathop{=}\limits^{\text{by \eqref{formula: a4}}}-v_{1,1,-1}$. Thus, $v_{1,1,-1}=0$. 

Suppose that $d\ge 3$. In the main relation \eqref{formula: main relation} the numbers $b_i$ and $a_j$ can be positive or negative. In Section~\ref{subsubsection: n=2 case, relations} we write explicitly the relations, that we have, depending on the numbers of positive $b_i$'s and $a_j$'s. Using the induction assumption we ignore the terms with the smaller degree. Then in Section~\ref{subsubsection: proof of first} we prove~\eqref{formula: first} and in Section~\ref{subsubsection: proof of second} we prove~\eqref{formula: second}.

\subsubsection{Relations}\label{subsubsection: n=2 case, relations}

Let $c_1,c_2,c_3,a$ be positive integers such that $c_1+c_2+c_3+a=d$. Then, if we put $b_1=c_1$, $b_2=c_2$, $b_3=c_3$, $b_4=-d$ and $a_1=a$ in \eqref{formula: main relation}, we get
\begin{gather}\label{formula: b1}
\sum_{\substack{\{i,j,k\}=\{1,2,3\}\\i<j}}v_{c_k,c_i+c_j,a}=0.
\end{gather}

Let $c_1,c_2,c_3,c_4,a$ be positive integers such that $c_1+c_2+a=c_3+c_4=d$. Putting $b_1=c_1$, $b_2=c_2$, $b_3=-c_3$, $b_4=-c_4$ and $a_1=a$ in \eqref{formula: main relation}, we get
$$
v_{-c_3,c_1+c_2,a}+z_{c_1,c_2,a}=0.
$$
Applying relations~\eqref{formula: a3} and \eqref{formula: a4}, we get
\begin{gather}\label{formula: b3}
z_{c_1,c_2,a}-z_{c_3,c_4,-a}=0.
\end{gather}

Let $c_1,c_2,c_3,c_4,a$ be positive integers such that $c_1+c_2+c_3=a+c_4=d$. Then, if we put $b_1=c_1$, $b_2=c_2$, $b_3=c_3$, $b_4=-c_4$ and $a_1=-a$ in \eqref{formula: main relation}, we get
\begin{gather}\label{formula: b2}
\sum_{\substack{\{i,j,k\}=\{1,2,3\}\\i<j}}v_{c_k,c_i+c_j,-a}=0.
\end{gather}

\subsubsection{Proof of \eqref{formula: first}}\label{subsubsection: proof of first} 

We have relations~\eqref{formula: a1} and~\eqref{formula: b1}, therefore, we can apply Lemma~\ref{lemma: linear system3}. By this lemma, the cases $d=3,4$ are done. Suppose $d\ge 5$, then there exists a class $\alpha\in R^{g-2}(\mcM_{g,n})/K$ such that
\begin{gather*}
v_{i,j,k}=\left(\frac{i}{d-1}-\frac{1}{3}\right)\left(\delta_{k,1}-\frac{1}{k}\delta_{j,1}\right)\alpha.
\end{gather*}
By \eqref{formula: b3}, we have $z_{d-2,1,1}=z_{d-3,2,1}$. Therefore, we get 
$$
0=z_{d-2,1,1}-z_{d-3,2,1}=-\left(\frac{1}{d-1}-\frac{1}{3}\right)\alpha+\frac{\alpha}{3}=\left(-\frac{1}{d-1}+\frac{2}{3}\right)\alpha.
$$
Since the coefficient $\left(-\frac{1}{d-1}+\frac{2}{3}\right)$ is not equal to zero, we have $\alpha=0$. This completes the proof of \eqref{formula: first}. 

\subsubsection{Proof of \eqref{formula: second}}\label{subsubsection: proof of second}
We have
\begin{gather}\label{formula: symmetry property}
v_{i,j,-k}\mathop{=}^{\text{by \eqref{formula: a2}}}v_{-(d-k),j,-k}\mathop{=}^{\text{by \eqref{formula: a1}}}\frac{j}{k}v_{-(d-k),-k,j}\mathop{=}^{\text{by \eqref{formula: a4}}}-\frac{j}{k}v_{d-k,k,-j}.
\end{gather}
As a consequence, we get
\begin{gather}\label{formula: equal}
v_{i,j,-j}=0.
\end{gather}

Suppose that $d=3$. Then we have only four classes $v_{1,2,-1}$, $v_{1,2,-2}$, $v_{2,1,-1}$ and $v_{2,1,-2}$. By \eqref{formula: equal}, $v_{1,2,-2}=v_{2,1,-1}=0$. By \eqref{formula: symmetry property}, $v_{2,1,-2}=-\frac{1}{2}v_{1,2,-1}$. Finally, by~\eqref{formula: b2}, $v_{1,2,-1}=0$.

Suppose that $d\ge 4$. From relations~\eqref{formula: b2} and Lemma \ref{lemma: linear system} it follows that, for any $1\le k\le d-1$, there exists a class $\alpha_k\in R^{g-2}(\mcM_{g,n})/K$ such that
\begin{gather*}
v_{i,j,-k}=\left(\frac{i}{d}-\frac{1}{3}\right)\alpha_k,\text{ if $2\le j\le d-1$}.
\end{gather*}
On the other hand, from \eqref{formula: b3} and \eqref{formula: first} it follows that $z_{i,j,-k}=0$, if $k\le d-2$. If $i,j\ge 2$, then $z_{i,j,-k}=-\frac{\alpha_k}{3}$. Thus, $\alpha_k=0$, if $1\le k\le d-2$. Therefore $v_{i,j,-k}=0$, if $j\ge 2$ and $k\le d-2$. Since $v_{d-1,1,-k}=-z_{d-1,1,-k}-v_{1,d-1,-k}$, we conclude that~$v_{i,j,-k}=0$, if~$k\le d-2$.

It remains to prove that $v_{i,j,-(d-1)}=0$. Applying \eqref{formula: symmetry property}, we get $v_{i,j,-(d-1)}=0$, if $j\le d-2$. Finally, by \eqref{formula: equal}, $v_{1,d-1,-(d-1)}=0$. This completes the proof of Proposition~\ref{proposition: second step} in the case $n=2$.

\subsubsection{Proof of Proposition~\ref{proposition: second step}: the case $n\ge 3$}\label{subsubsection: n bigger than 2}

From Lemma~\ref{lemma: span} it follows that it is sufficient to prove that $V_g\left(\prod_{i=1}^{n+1}m_{a_i}\right)\in K, \text{if $a_i\ge 1$}$. 

We proceed by induction on the degree $d=\sum_{i=1}^{n+1}a_i$. The smallest possible degree is~$n+1$, then $V_g\left(\prod_{i=1}^{n+1}m_1\right)\in K$, because of the property~\eqref{formula: 2,i-symmetry}. 

Suppose that $d\ge n+2$. From relations \eqref{formula: main relation}, \eqref{formula: 2,i-symmetry} and Lemma \ref{lemma: linear system3} it follows that $V_g\left(\prod_{i=1}^{n+1}m_{a_i}\right)\in K$, if $a_2\ge 2$ and $a_3\ge 2$. Applying \eqref{formula: 2,i-symmetry} and \eqref{formula: i,j-symmetry}, we get that $V_g\left(\prod_{i=1}^{n+1}m_{a_i}\right)\in K$, if there exist $n+1\ge j>i\ge 2$ such that $a_i,a_j\ge 2$. Suppose there exists at most one $i\ge 2$ such that $a_i\ge 2$. Since $n\ge 3$, there exist $n+1\ge j>k\ge 2$ such that $a_j=a_k=1$. If we again apply \eqref{formula: 2,i-symmetry} or \eqref{formula: i,j-symmetry}, we get $V_g\left(\prod_{i=1}^{n+1}m_{a_i}\right)\in K$. This completes the proof of Proposition~\ref{proposition: second step} in the case $n\ge 3$. 


\bigskip
\footnotesize
\noindent\textit{Acknowledgments.} The authors are grateful to C. Faber, M. Kazarian and R. Pandharipande for useful discussions. We would like to thank the anonymous referee for valuable remarks and suggestions that allowed us to improve the exposition of this paper.

A.~B. was supported by grant ERC-2012-AdG-320368-MCSK in the group of R. Pandharipande at ETH Zurich, by a Vidi grant of the Netherlands Organization for Scientific Research, Russian Federation Government grant no. 2010-220-01-077 (ag. no. 11.634.31.0005), the grants RFFI 13-01-00755, NSh-4850.2012.1, the Moebius Contest Foundation for Young Scientists and "Dynasty" foundation.

S.~S. was supported by a Vidi and a Vici grants of the Netherlands Organization for Scientific Research.

D.~Z. was supported by the grant ANR-09-JCJC-0104-01.

\end{document}